\documentclass[12pt,a4paper,reqno]{amsart}
\usepackage[utf8]{inputenc}
\usepackage[T1]{fontenc}
\usepackage{amsmath}
\usepackage{amsthm}
\usepackage{amssymb}
\usepackage[abbrev,lite]{amsrefs}
\usepackage{mathrsfs}
\usepackage[dvipsnames]{xcolor}
\usepackage{bm}
\usepackage{enumitem}
\usepackage{hyperref}
\DefineSimpleKey{bib}{myurl}
\newcommand\myurl[1]{\url{#1}}
\BibSpec{webpage}{%
+{}{\PrintAuthors} {author}
+{,}{ \textit} {title}
+{}{ \parenthesize} {date}
+{,}{ \myurl} {myurl}
+{,}{ } {note}
+{.}{ } {transition}
}
\AtBeginDocument{\def\MR#1{}}
\makeatletter
\@namedef{subjclassname@2020}{%
\textup{2020} Mathematics Subject Classification}
\makeatother
\numberwithin{equation}{section}
\allowdisplaybreaks

\newtheorem{thm}{}[section]
\newtheorem{theorem}[thm]{Theorem}
\newtheorem{corollary}[thm]{Corollary}
\newtheorem{lemma}[thm]{Lemma}

\theoremstyle{definition}
\newtheorem{example}[thm]{Example}
\newtheorem{question}[thm]{Question}
\newcommand{\abs}[1]{\left\lvert#1\right\rvert}
\newcommand{\norm}[1]{\left\lVert#1\right\rVert}
\newcommand{\enbrace}[1]{\left\lbrace#1\right\rbrace}
\newcommand{\enbrak}[1]{\left[#1\right]}
\newcommand{\enpar}[1]{\left(#1\right)}
\DeclareMathOperator{\Sz}{Sz}
\DeclareMathOperator*{\Dz}{Dz}
\DeclareMathOperator{\supp}{supp}
\DeclareMathOperator{\spn}{span}
\DeclareMathOperator{\codim}{codim}
\newcommand{\LP}{\ensuremath{\bm{\Lambda}}}
\newcommand{\Id}{\ensuremath{\mathrm{Id}}}

\newcommand{\HH}{\ensuremath{\mathbb{H}}}
\newcommand{\VV}{\ensuremath{\mathbb{V}}}
\newcommand{\YY}{\ensuremath{\mathbb{Y}}}
\newcommand{\XX}{\ensuremath{\mathbb{X}}}

\newcommand{\NN}{\ensuremath{\mathbb{N}}}
\newcommand{\ZZ}{\ensuremath{\mathbb{Z}}}

\newcommand{\FF}{\ensuremath{\mathbb{F}}}
\newcommand{\DD}{\ensuremath{\mathbb{D}}}
\newcommand{\BB}{\ensuremath{\mathbb{B}}}
\newcommand{\UU}{\ensuremath{\mathbb{U}}}
\newcommand{\xx}{\ensuremath{\bm{x}}}
\newcommand{\yy}{\ensuremath{\bm{y}}}

\newcommand{\uu}{\ensuremath{\bm{u}}}
\newcommand{\vv}{\ensuremath{\bm{v}}}
\newcommand{\UL}{\ensuremath{\bm{U}}}
\newcommand{\LL}{\ensuremath{\bm{L}}}

\newcommand{\XB}{\ensuremath{\mathcal{X}}}
\newcommand{\Xt}{\ensuremath{\mathfrak{X}}}
\newcommand{\Yt}{\ensuremath{\mathfrak{Y}}}
\newcommand{\Zt}{\ensuremath{\mathfrak{Z}}}
\newcommand{\YB}{\ensuremath{\mathcal{Y}}}
\newcommand{\UB}{\ensuremath{\mathcal{U}}}
\newcommand{\Ut}{\ensuremath{\mathfrak{U}}}
\newcommand{\VB}{\ensuremath{\mathcal{V}}}
\newcommand{\Rt}{\ensuremath{\mathcal{R}}}
\newcommand{\Nt}{\ensuremath{\mathcal{N}}}
\newcommand{\Mt}{\ensuremath{\mathcal{M}}}
\newcommand{\Et}{\ensuremath{\mathcal{E}}}
\newcommand{\Ts}{\ensuremath{\mathcal{T}}}
\newcommand{\Jt}{\ensuremath{\mathcal{J}}}
\newcommand{\Ct}{\ensuremath{\mathcal{C}}}
\newcommand{\Bt}{\ensuremath{\mathcal{B}}}

\author[J. L. Ansorena]{Jos\'e L. Ansorena}\address{Department of Mathematics and Computer Sciences\\
Universidad de La Rioja\\
Logro\~no 26004\\ Spain}
\email{joseluis.ansorena@unirioja.es}
\author[G. Bello]{Glenier Bello}
\address{Departamento de Matem\'{a}ticas e Instituto Universitario de Matem\'{a}ticas y Aplicaciones\\
Universidad de Zaragoza\\
50009 Zaragoza\\
Spain}
\email{gbello@unizar.es}
\subjclass[2020]{46B03, 46B07, 46B10, 46B15, 46B20, 46B25, 46B42, 46B08, 46E30, 46E40}
\keywords{Lebesgue spaces, sequence spaces, isomorphic classification of Banach spaces, linear dimension of Banach spaces}
\begin{document}
\title[]{Embeddability of $\ell_p$-spaces into mixed-norm Lebesgue spaces in connection with the validity of vector-valued extensions of the Riesz--Fischer Theorem}
\begin{abstract}
The aim of this paper is twofold. On the one hand, we compute, in terms of $r$ and $s$, the indices $p$ for which $\ell_p$ isomorphically embeds into the mixed-norm separable spaces $L_s(L_r)$, $\ell_s(L_r)$, $L_s(\ell_r)$ and $\ell_s(\ell_r)$. On the other hand, we use this information to move forward in the isomorphic classification of mixed-norm spaces. In particular, we tell apart the spaces $L_2(L_r)$ and $\ell_2(L_r)$, $r\not=2$.
\end{abstract}
\thanks{Both authors acknowledge the support of the Spanish Ministry for Science and Innovation under Grant PID2022-138342NB-I00 for \emph{Functional Analysis Techniques in Approximation Theory and Applications (TAFPAA)}. G. Bello has also been partially supported by PID2022-137294NB-I00, DGI-FEDER and by Project E48\_23R, D.G. Arag\'{o}n}
\maketitle
\section{Introduction}\noindent
One of the driving forces that made functional analysis move forward is the study of whether established results in spaces of functions have a vector extension for functions with values in a Banach space. Sometimes theorems are valid for every Banach space. For instance, a fundamental result in mathematical analysis is the validity of Lebesgue integration for Banach-valued functions achieved by S.\ Bochner \cite{Bochner1933}. Other times theorems cannot be extended to every Banach space, which naturally leads to consider the class of spaces for which the extension holds. These classes of Banach spaces consist of spaces that behave well in some sense, so Hilbert spaces should be an element of all of them. Possibly, the more important property that arises this way is the \emph{Radon--Nikodym property}. This property, which goes back to the above-mentioned paper \cite{Bochner1933}, defines the Banach spaces $\XX$ such that every $\XX$-valued absolutely continuous measure is a density measure. Another important property is the \emph{unconditional martingale difference} (UMD for short) property introduced by Maurey \cite{Maurey1975} and Pisier \cite{Pisier1975}. We say that $\XX$ is a UMD space if there are $p$, $C\in(1,\infty)$ such that
\[
\norm{ \sum_{n=1}^m \varepsilon_n \enpar{f_n-f_{n-1}} }_{L_p(\XX)}
\le C
\norm{f_m}_{L_p(\XX)}
\]
for every $m\in\NN$, every sequence of signs $(\varepsilon_n)_{n=1}^\infty$ and every martingale $(f_n)_{n=0}^\infty$ with $f_0=0$.

The integration theory constructed by Lebesgue in his 1902 doctoral thesis boosted the development of functional analysis in the early 20th century. A cornerstone result achieved by the pioneers in these early years is the Radon--Nikodym theorem which later led to the definition of the Radon--Nikodym property. The Riesz--Fischer theorem is another one. Although the literature contains many different formulations of this result, all of them are essentially equivalent. A common form of Riesz--Fischer theorem states that the spaces $L_2$ and $\ell_2$ are, via an orthonormal basis of $L_2$, isometrically isomorphic. So, the following question naturally arises when aiming at Banach-valued extensions of classical theorems.

\begin{question}\label{qt:general}
Let $\XX$ be a Banach space. Is $L_2(\XX)$ isomorphic to $\ell_2(\XX)$?
\end{question}

This question was explicitity posed in \cite{Diestel1977}, where the following example that dashes the hope of giving a positive answer in the general case is provided.

\begin{example}[J.\@ Diestel in \cite{Diestel1977}]\label{ex:Diestel}
D.\@ Aldous \cite{Aldous1979} proved that if $L_p(\XX)$ has an unconditional basis, then $1<p<\infty$ and $\XX$ is superreflexive. Therefore, the answer to Question~\ref{qt:general} is negative for any non superreflexive Banach space with an unconditional basis. Instances of such spaces are the Hardy space $H_1(\DD)$ over the unit disc (see \cite{Maurey1980}), the (unique up to isomorphism) Banach space $\UU$ with an unconditional basis that is complementarily universal for separable spaces with an unconditonal basis (see \cite{Pel1969}), the original Tsirelson's space $\Ts^*$ and its dual $\Ts$ (see \cite{Tsirelson1974}), and mixed-norm spaces
\[
Z_{p,q}=\ell_q(\ell_p), \, B_{p,q}=\enpar{\oplus_{n=1}^\infty \ell_p^n}_{\ell_q}, \quad p,q\in[1,\infty],
\]
(we replace $\ell_q$ with $c_0$ if $q=\infty$) in the case when $\{p,q\}\cap\{1,\infty\}$ is nonempty. In particular, the answer to Question~\ref{qt:general} is negative for $c_0$ and $\ell_1$.
\end{example}

We also mention the discussion that took place in the MathOverflow post \cite{Gonzalez2015}. M.\@ Gonz\'{a}lez noted that, since $L_2(L_2)$ and $\ell_2(L_2)$ are lattice isomorphic, the answer is positive for $\XX=L_2(\YY)$, where $\YY$ is an arbitrary Banach space. Let $\BB$ be the (unique up to isomorphism) separable Banach space with the bounded approximation property (BAP, for short) that is complementarily universal for separable spaces with the BAP (see \cites{Pel1969,JRZ1971}). Since $L_2(\BB)$ is a separable Banach space with the BAP, $L_2(\BB)$ is isomorphic to $\BB$ by the Pe{\l}czy\'{n}ski decomposition technique. Hence, as W.\@ Johnson pointed out, the answer to Question~\ref{qt:general} is positive for $\BB$. For spaces for which the answer is negative, let us record the examples given in the above-mentioned MathOverflow post.

\begin{example}[P.\@ Brooker in \cite{Gonzalez2015}]
Through recent years, various ordinal indices have become relevant to studying the geometry of Banach spaces. We consider here the \emph{Szlenk index} $\Sz(\XX)$ and the weak*-dentability index $\Dz(\XX)$ of a Banach space $\XX$. It is known \cites{Brooker2011,Lancien2006} that
\[
\Sz(\ell_2(\XX))=\Sz(\XX)\le \Dz(\XX)\le\Sz(L_2(\XX))
\]
for every Banach space $\XX$.

Since a Banach space $\XX$ is superreflexive if and only if $\Dz(\XX)\le \omega$ \cite{Lancien1995}, the answer to Question~\ref{qt:general} is negative for any non superreflexive Banach space $\XX$ with $\Sz(\XX)=\omega$. Let us record some instances of this kind of spaces.
\begin{itemize}[leftmargin=*]
\item Let $1<p<\infty$ and $\Xt:=(\XX_n)_{n=1}^\infty$ be a sequence of finite dimensional Banach spaces. Set
\[
\XX=\enpar{\oplus_{n=1}^\infty \XX_n}_{\ell_p}.
\]
On the one hand, $\Sz(\XX)=\omega$. On the other hand, $\XX$ fails to be superreflexive as long as either $\ell_1$ or $\ell_\infty$ is finitely representable in $\Xt$.

\item Aside from $c_0$, several preduals of $\ell_1$ constructed by Bourgain and Delbaen \cite{BD1980} have Szlenk index equal to $\omega$ \cite{Alspach2000}.

\item The quasi-reflexive James space $\Jt$ satisfies $\Sz(\Jt)=\omega$ \cite{Lancien1990}.
\end{itemize}

Certain compact Hausdorff spaces $K$ satisfies $\Sz(\Ct(K))<\Dz(\Ct(K))$. Indeed if $\alpha$ is an ordinal number with $\omega^{\omega^n}\le \alpha <\omega^{\omega^{n+1}}$ for some $n\in\NN$, and we set $K=[0,\alpha]$ equipped with the order topology, then $\Sz(\Ct(K))=\omega^{n+1}$ \cite{Samuel1983} and $\Dz(\Ct(K))=\omega^{n+2}$ \cite{HLP2009}.
\end{example}

\begin{example}[M.\@ Ostrovskii in \cite{Gonzalez2015}]
On the one hand, the Banach space $\ell_2(\XX)$ inherits the Banach--Saks property from $\XX$ \cite{Partington1977}. On the other hand, J.\@ Bourgain \cite{Guerre1980} and W.\@ Schachermayer \cite{Schachermayer1981} constructed Banach spaces $\XX$ with the Banach--Saks property such that $L_2(\XX)$ does not have it.
\end{example}

\begin{example}[T.\@ Kania in \cite{Gonzalez2015}]\label{ex:Kania}
Fix $1\le s<\infty$. On the one hand, the Banach space $\ell_s(\XX)$ is a Grothendieck space provided $\XX$ is. On the other hand, if $\XX$ is a non-reflexive Banach space, then $L_s(\XX)$ is not a Grothendieck space \cite{Diaz1995}. Hence, the answer to Question~\ref{qt:general} is negative for non-reflexive Banach spaces with the Grothendieck property. Instances of such spaces are (see \cite{GonzalezKania2021})
\begin{itemize}
\item $\Ct(K)$ for any Stonean compact space $K$,
\item $L_\infty(\mu)$ for any measure space $\mu$,
\item the Hardy space $H_\infty(\DD)$ over the unit disc, and
\item the algebra $\Bt(\HH)$ of bounded operators on a Hilbert space $\HH$.
\end{itemize}
\end{example}

Notice that all above-mentioned counterexamples are non superreflexive Banach spaces. To record a superreflexive Banach space $\XX$ for which the answer to Question~\ref{qt:general} is negative, it will be convenient to introduce a terminology that will be deeply used throughout the paper. Namely, we will use the symbol $ \YY \sqsubseteq\XX$ for standing that the quasi-Banach space $\YY$ isomorphically embeds into the quasi-Banach space $\XX$.

\begin{example}[S.\@ Dilworth in \cite{Dilworth1990}]\label{ex:Dilworth}
Set
\[
M_p=
\begin{cases}
L_p((0,\infty))\cap L_2((0,\infty)) & \mbox{ if } 0<p\le 2, \\ L_p((0,\infty)) + L_2((0,\infty)) & \mbox{ if } p\ge 2.
\end{cases}
\]
On the one hand, $M_p\sqsubseteq L_2(\ell_p)$ (see \cite{Dilworth1990}*{Theorem 3.1}). On the other hand, $M_p\not\sqsubseteq \ell_2(\ell_p)$ unless $p=2$ (see \cite{Dilworth1990}*{Theorem 3.7}). Hence $L_2(\ell_p)$ and $\ell_2(\ell_p)$ fail to be isomorphic for all $p\in(0,\infty)\setminus\{2\}$.
\end{example}

In this note, we contribute to the theory by providing a family of spaces that contains $\ell_p$ and $L_p$, $1<p<\infty$, $p\not=2$, for which the answer to Question~\ref{qt:general} is negative.

The only feasible strategy for proving that two given Banach spaces are not isomorphic seems to be finding a property that distinguishes them. The feature of Banach spaces that we will use to prove that $L_2(L_p)$ and $\ell_2(L_p)$ are not isomorphic is their structure of basic sequences. Specifically, we will study the embeddability of $\ell_p$-spaces into mixed-norm Lebesgue spaces. In this terminology, given a quasi-Banach space $\XX$ over the real or complex field $\FF$ we set
\[
\LP(\XX)=\enbrace{p\in[1,\infty] \colon \ell_p \sqsubseteq \XX},
\]
with the convention that $\ell_\infty$ means here $c_0$. This set is a feature of $\XX$ whose study goes back to Banach's book \cite{Banach1932}, where it is framed within the study of the `linear dimension' of Banach spaces. Let us mention, for instance, that proving that $\LP(\XX)$ can be empty \cite{Tsirelson1974} was one of the milestones of the theory of Banach spaces.

In this terminology, we will prove the following.
\begin{theorem}\label{thm:mainRF}
Let $\XX$ be a Banach space. Suppose that there are $2<p<r<\infty$ such that either $\YY=\XX$ or $\YY=\XX^*$ satisfies
$r\in\LP(\YY)$ and $p\notin\LP(\YY)$. Then $L_2(\XX)$ and $\ell_2(\XX)$ are not isomorphic.
\end{theorem}

The structure of the paper is as follows. Section~\ref{sect:main} is geared towards proving Theorem~\ref{thm:mainRF}. The subsequent Section~\ref{sect:lpembeds} is devoted solely to studying $\LP(\XX)$ when $\XX$ is a mixed-norm space. We close the paper with Section~\ref{sect:iso}, where applications to the isomorphic theory of Banach spaces are given.

We close this introductory section by setting some terminology that will be heavily used. We denote by $S_\XX$ the unit sphere of a quasi-Banach space $\XX$. Given a family $\XB=(\xx_n)_{n\in\Nt}$ in $\XX$, $[\XB]=[\xx_n \colon n\in\Nt]$ stands for its closed linear span. If $[\XB]=\XX$ we say that $\XB$ is \emph{complete} within $\XX$. If
\[
\inf_ {n\in\Nt} \norm{\xx_n}>0, \quad \sup_ {n\in\Nt} \norm{\xx_n}<\infty,
\]
we say that $\XB$ is \emph{semi-normalized}. Given another family $\YB=(\yy_n)_{n\in\Nt}$ in a quasi-Banach space $\YY$, we say that $\XB$ and $\YB$ are \emph{equivalent} if there is an isomorphic embedding $T\colon[\XB]\to \YY$ such that $T(\xx_n)=\yy_{n}$ for all $n\in\Nt$. The symbol $\XX\simeq\YY$ will mean that the quasi-Banach spaces $\XX$ and $\YY$ are isomorphic.

Given two elements $n$ and $k$, we denote by $\delta_{n,k}$ its Kronecker delta given by $\delta_{n,k}=1$ if $n=k$ and $\delta_{n,k}=0$ otherwise. Given family $\XB=(\xx_n)_{n\in\Nt}$ in a quasi-Banach space $\XX$, we say that $\XB^*=(\xx_n^*)_{n\in\Nt}$ in $\XX^*$ is family of \emph{coordinate functionals} for $\XB$ within $\XX$ if $\xx_n^*(\xx_k)=\delta_{n,k}$ for all $n$, $k\in\Nt$. A \emph{minimal system} of $\XX$ is a family in $\XX$ for which there exists a family of coordinate functionals within $\XX$.

A quasi-Banach lattice is \emph{L-concave} if it has some nontrivial lattice concavity. It is known \cite{Kalton1984b} that any L-concave quasi-Banach lattice has some nontrivial lattice convexity. For convenience, we shall only deal with quasi-Banach lattices arising from function quasi-norms over measure spaces. We refer the reader to \cite{AnsorenaBello2022} for the basics of this generalization of function spaces and function norms. We say that a quasi-Banach function space $\LL$ is \emph{absolutely continuous} if it is built from an absolutely continuous function quasi-norm. If $\LL$ is L-concave, then it is absolutely continuous.

The modulus of concavity of a quasi-Banach space $\XX$ is the smallest constant $\kappa=\kappa(\XX)$ such that
\[
\norm{f+g} \le\kappa \enpar{ \norm{f}+\norm{g}}, \quad f,g\in\XX.
\]
Let $\LL$ be a quasi-Banach function space over a countable set $\Nt$ and $\Xt=(\XX_n)_{n\in\Nt}$ be a sequence of quasi-Banach spaces. If $\sup_{n\in\Nt}\kappa(\XX_n)<\infty$, then
\[
\LL(\Xt)=\enpar{\oplus_{n=1}^\infty \XX_n}_{\LL}
\]
is a quasi-Banach space. If there is a quasi-Banach space $\XX$ such that $\XX_n=\XX$ for all $n\in\Nt$, we put $\LL(\Xt)=\LL(\XX)$.

Given a measure space $(\Omega,\Sigma,\mu)$ and a quasi-Banach space $\XX$, the support of a measurable function $f\colon \Omega\to \XX$ will be the set
\[
\enbrace{\omega\in\Omega \colon f(\omega)\in\XX\setminus\{0\} }.
\]
In this regard, we emphasize that given quasi-Banach function spaces $\LL_i$ over a measures spaces $(\Omega_i,\Sigma_i,\mu_i)$, $i=1$, $2$, we regard $\LL_1(\LL_2)$ as a quasi-Banach function space over the product measure space $(\Omega_1\times\Omega_2, \Sigma_1\otimes\Sigma_2, \mu_1\otimes\mu_2)$. So, the support of $f\in \LL_1(\LL_2)$ will be a measurable subset of $\Omega_1\times\Omega_2$. If $\LL_1$ and $\LL_2$ are L-concave, so is $\LL_1(\LL_2)$.

Given $0<p<\infty$ and a measure space $(\Omega,\Sigma,\mu)$, the Lebesgue space $L_p(\mu)$ is $p$-convex and $p$-concave. In particular, it is absolutely continuous.

We say that a quasi-Banach space $\XX$ is \emph{$C$-complemented} in a quasi-Banach space $\YY$ if there are linear operators $J\colon\XX\to\YY$ and $P\colon\YY\to \XX$ such that $P\circ J=\Id_{\XX}$ and $\norm{J}\norm{P} \le C$. Given countable infinite families of quasi-Banach spaces $\Xt=(\XX_n)_{n\in\Nt}$ and $\Yt=(\YY_n)_{m\in\Mt}$, we say that \emph{$\Xt$ is complemented in $\YB$}, and we put $\Xt \trianglelefteq \Yt$, if there is a constant $C\in[1,\infty)$ such that for all $n\in\Nt$ and all $F\subseteq\Mt$ finite there is $m\in\Mt\setminus J$ such that $\XX_n$ is $C$-complemented in $\YY_m$. If there is a Banach space $\UU$ such that $\XX_n=\UU$ for all $n\in\Nt$ (resp., $\YY_m=\UU$ for all $m\in\Mt$) we replace $\Xt$ (resp., $\Yt$) with $\UU$ in the symbol $\Xt \trianglelefteq \Yt$.

We denote by $\NN_0$ the set of all nonnegative integers, that is, $\NN_0=\NN\cup\{0\}$.
\section{The \texorpdfstring{$L_p$}{}-valued Riesz--Fischer theorem does not hold}\label{sect:main}\noindent
For exponencial ease, we shall record several known results. The first of them is an extension to the quasi-Banach setting of a classical result concerning unconditional basic sequences in Banach lattices (see \cite{Maurey1974}*{Lemme 5 and Lemme 6} or \cite{LinTza1979}*{Theorem 1.d.6}).

\begin{theorem}[see \cite{AlbiacAnsorena2025}*{Lemma 2.5}]\label{thm:AA}
Let $(\xx_n)_{n\in\Nt}$ be an unconditional basic sequence in an L-concave quasi-Banach function space $\LL$. Then, there is a constant $C$ such that
\[
\frac{1}{C} \norm{f}_{\LL} \le \norm{\enpar{\sum_{n\in\Nt} \abs{a_n}^2 \abs{x_n}^2}^{1/2}}_{\LL} \le C \norm{f}_{\LL}
\]
for all $f=\sum_{n\in\Nt} a_n\, \xx_n\in\LL$.
\end{theorem}

Theorem~\ref{thm:AA} allows us to regard unconditional basic sequences in quasi-Banach lattices as families in Hilbert-valued lattices. This gaze yields the following result.

\begin{theorem}\label{lem:AANew}
Let $\XB=(\xx_n)_{n\in\Nt}$ be an unconditional basic sequence in an L-concave quasi-Banach function space $\LL$. Then, $\XB$ is equivalent to a disjointly supported sequence in $\LL(\ell_2(\Nt))$.
\end{theorem}

\begin{proof}
Suppose that $\LL$ is a quasi-Banach function space over a measure space $(\Omega,\Sigma,\mu)$. If we define for each $n\in\Nt$ $\yy_n\colon\Omega\times\Nt \to\FF$ by
\[
\yy_n(\omega,k)=\xx_n(\omega) \delta_{n,k}, \quad \omega\in\Omega, \, k\in\Nt,
\]
then $\supp(\yy_n) \subset A_n:=\Omega\times \{n\}$. It is clear that $(A_n)_{n=1}^\infty$ is a partition of $\Omega\times\Nt$, and
\[
\norm{\sum_{n\in\Nt} a_n\, \yy_n}_{\LL(\ell_2(\Nt))} = \norm{\enpar{\sum_{n\in\Nt} \abs{a_n}^2 \abs{\xx_n}^2}^{1/2}}_{\LL},
\quad (a_n)_{n\in\Nt}\in c_{00}(\Nt).
\]
So, the result follows from Theorem~\ref{thm:AA}.
\end{proof}

One of the techniques we will use is the small perturbation principle. We next record and prove the precise statement of this principle we will apply.

\begin{lemma}\label{lem:SPP}
Let $\XB=(\xx_n)_{n\in\Nt} $ and $\YB=(\yy_n)_{n\in\Nt}$ be families in a $q$-Banach space $\XX$, $0<q\le 1$. Suppose that $\YB$ is a minimal system of $\YY=[\yy_n \colon n\in\Nt]$ with coordinate functionals $(\yy_n^*)_{n\in\Nt}$ (within $\YY$). If
\[
\sum_{n\in\Nt} \norm{\yy_n^*}^q \norm{\xx_n-\yy_n}^q<1,
\]
then $\XB$ is equivalent to $\YB$.
\end{lemma}

\begin{proof}
The linear operator $S\colon \YY \to \XX$ given by
\[
S(f)=\sum_{n=1}^\infty \yy_n^*(f) \enpar{\xx_n-\yy_n}, \quad f\in\YY,
\]
is well-defined, and satisfies $\norm{S}<1$. Let $J$ be be inclusion of $\YY$ into $\XX$. Set $T=J-S$. Since $T(\yy_n)=\xx_n$ for all $n\in\Nt$ and
\[
\enpar{1-\norm{S}^q} \norm{f}^q \le \norm{T(f)}^q \le \enpar{1+\norm{S}^q} \norm{f}^q, \quad f\in\YY,
\]
we are done.
\end{proof}

Banach \cite{Banach1932} conjectured that $L_p$ isometrically embeds into $L_r$ for any $1\le r \le p \le 2$. This question was solved in the positive in \cite{BDCK1966} (see also \cite{LinPel1968}). We will use the extension to the quasi-Banach setting of this embedding.

\begin{theorem}[\cite{Kanter1973}]\label{thm:AMS}
Let $0<r\le 2$. Then $L_p$ isometrically embeds into $L_r$ for all $p\in[r,2]$.
\end{theorem}

Combining Theorem~\ref{thm:AMS} with previous results achieved by Paley \cite{Paley1936} (see also \cite{KadPel1962}) settles the embeddability of $\ell_p$-spaces into $L_r$-spaces.

\begin{corollary}\label{cor:AMS+Paley}
Given $0<r\le 2$, $\LP(L_r)=[r,2]$, while $\LP(L_r)=\{2,r\}$ for all $2\le r<\infty$.
\end{corollary}

In turn, the embeddability between $\ell_p$-spaces was settled by Pe{\l}czy\'{n}ski \cite{Pel1960} in the case when $p\ge 1$, and extended to the whole range $p>0$ by Stiles \cite{Stiles1970}. We will use an easy generalization of this Stiljes's result.

\begin{theorem}\label{thm:PelSti}
Let $0<s<\infty$, $0<q\le 1$, and $(\XX_n)_{n=1}^\infty$ be a sequence of nonnull finite-dimensional $q$-Banach spaces. Set $\XX=\enpar{\oplus_{n=1}^\infty \XX_n}_{\ell_s}$. Then, $\LP(\XX)=\{s\}$.
\end{theorem}

\begin{proof}
Note that $\XX$ is a $r$-Banach space, where $r=\min\{q,s\}$. Throughout this proof, the support of $x=(x_n)_{n=1}^\infty\in\XX$ will be the set
\[
\enbrace{n\in\NN \colon x_n\not=0}.
\]

Let $0<p\le \infty$ and $\YB=(\yy_k)_{k=1}^\infty$ be a sequence in $\XX$ equivalent to the unit vector system of $\ell_p$ ($c_0$ if $p=\infty$). Since the unit ball of $\XX_n$ is compact for all $n\in\NN$, passing to a subsequence we can assume that $\YB$ converges coordinate-wise. Replacing $\YB$ with $(\yy_{2k-1}-\yy_{2k})_{n=1}^\infty$ we can assume that $\YB$ converges to zero coordinate-wise. Since the unit vector system is a minimal system of $\ell_p$, $\YB$ is a minimal system of $\YY=[\YB]$. Let $(\yy_k^*)_{k=1}^\infty$ be the coordinate functionals for $\YB$ within $\YY$. Pick a sequence $(\varepsilon_k)_{k=1}^\infty$ in $(0,\infty)$ with $\sum_{k=1}^\infty \norm{\yy_k^*}^r \varepsilon_k^r <1$. By the gliding-hump technique, passing to a subsequence we can assume that there is a disjointly supported sequence $\XB=(\xx_k)_{k=1}^\infty$ in $\XX$ such that $\norm{\yy_k-\xx_k}\le\varepsilon_k$ for all $k\in\NN$. By Lemma~\ref{lem:SPP}, $\YB$ is equivalent to $\XB$. In turn, $\XB$ is equivalent to the unit vector system of $\ell_s$. Hence $p=s$.
\end{proof}

The following result was tailored by the authors of \cite{CembranosMendoza2011} to prove that the Banach spaces $\ell_p(\ell_q)$, $p$, $q\in[1,\infty]$, are pairwise non-isomorphic. Although they only dealt with the case $s\ge 1$, their proof extends verbatim to the case when $s>0$.

\begin{theorem}[\cite{CembranosMendoza2011}*{Corollary 2.2}]\label{thm:CM}
Let $0< s\le \infty$, and use the convention that $\ell_\infty$ means $c_0$. Let $\XX$ and $\YY$ be quasi-Banach spaces. If $\XX^2 \simeq \XX$ and $\YY \sqsubseteq \ell_s(\XX)$, then either $\YY \sqsubseteq \XX$ or $s\in\LP(\YY)$.
\end{theorem}

Given $t\in(0,\infty)$ and a quasi-Banach function space $\LL$ built from a function quasi-norm $\rho$ over a measure space $(\Omega,\Sigma,\mu)$, we denote by $\LL^{(t)}$ the quasi-Banach function space constructed from the function quasi-norm defined by
\[
f\mapsto \enpar{\rho(f^t)}^{1/t}, \quad f\in L_0^+(\mu).
\]
We call $\LL^{(t)}$ the \emph{$t$-convexification} of $\LL$.

\begin{lemma}\label{lem:DisjSup}
Let $\LL$ be quasi-Banach function space and $\UL$ be a quasi-Banach function space over a countable set $\Nt$. Suppose there is a disjointly supported family $(f_n)_{n\in\Nt}$ in $\LL$ equivalent to the unit vector system of $\UL$. Then, $\UL^{(t)} \sqsubseteq \LL^{(t)}$ for all $0<t<\infty$.
\end{lemma}

\begin{proof}
The mapping
\[
(a_n)_{n\in\Nt} \mapsto \sum_{n\in\Nt} a_n \abs{f_n}^{1/t}.
\]
defines an isomorphic embedding of $\UL^{(t)}$ into $\LL^{(t)}$.
\end{proof}

We are now ready to give the embeddability result that will allow us to prove Theorem~\ref{thm:mainRF}.

\begin{theorem}\label{lem:lpEmbeds}
Let $0<s<r<\infty$. Then, $[s,r] \subset \LP(L_s(\ell_r))$.
\end{theorem}

\begin{proof}
Fix $p\in[s,r]$. Pick $t\in(0,\infty)$ such that $tr=2$. Since $ts \le tp \le tr$, $\ell_{tp} \sqsubseteq L_{st}$ by Theorem~\ref{thm:AMS}. Hence, by Theorem~\ref{lem:AANew}, $L_{ts}(\ell_2)$ has a disjointly supported sequence equivalent to the unit vector system of $\ell_{tp}$. By Lemma~\ref{lem:DisjSup}, $\ell_{qtp} \sqsubseteq L_{qts} (\ell_{q2})$ for all $0<q<\infty$. Choosing $q=1/t$, we are done.
\end{proof}

\begin{proof}[Proof of Theorem~\ref{thm:mainRF}]
Assume by contradiction that $L_2(\XX)\simeq \ell_2(\XX)$. If $\YY=\XX^*$, then, by duality,
\[
L_2(\YY) \sqsubseteq (L_2(\XX))^*\simeq (\ell_2(\XX))^*\simeq \ell_2(\YY).
\]
Therefore, $L_2(\YY) \sqsubseteq \ell_2(\YY)$ in any case. Consequently, $L_2(\ell_r) \sqsubseteq \ell_2(\YY)$. By Theorem~\ref{lem:lpEmbeds}, $p\in\LP(\ell_2(\YY))$. Hence, by Theorem~\ref{thm:CM}, $2\in \LP(\ell_p)$. Since this assertion conflicts with Theorem \ref{thm:PelSti}, we are done.
\end{proof}

\begin{corollary}
Let $1<p<\infty$, $p\not=2$. Then, $\ell_2(L_p)\not\simeq L_2(L_p)$, and $\ell_2(\ell_p)\not\simeq L_2(\ell_p)$.\end{corollary}

\begin{proof}
Just combine Theorem~\ref{thm:mainRF} with Corollary~\ref{cor:AMS+Paley} or Theorem~\ref{thm:PelSti}.
\end{proof}
\section{Embeddings of \texorpdfstring{$\ell_p$}{}-spaces into mixed-norm spaces}\label{sect:lpembeds}\noindent
We will compute $\LP(\XX)$ in the case when $\XX$ is $L_s(L_r)$, $\ell_s(L_r)$, $L_s(\ell_r)$, $\ell_s(\ell_r)$ or $B_{r,s}$ for $r$, $s\in(0,\infty)$. We start with three consequences of Corollary~\ref{cor:AMS+Paley}, Theorem~\ref{thm:PelSti} and Theorem~\ref{thm:CM} whose straightforward proofs me omit.

\begin{theorem}\label{thm:LPBesov}
Let $0 <r\le \infty$ and $0<s< \infty$. Then $\LP(B_{r,s})=\{s\}$.
\end{theorem}

\begin{theorem}[cf.\ \cite{CembranosMendoza2011}*{Proposition 2.3}]\label{thm:lplq}
$\LP(\ell_s(\ell_r))=\{r,s\}$ for all $r$, $s\in (0,\infty)$.
\end{theorem}

\begin{theorem}\label{thm:lpLq}
Let $r$, $s\in (0,\infty)$.
\begin{itemize}
\item If $r\le s \le 2$, then $\LP(\ell_s(L_r))=[r,2]$.
\item If $r\le 2 \le s$, then $\LP(\ell_s(L_r))=[r,2] \cup\{s\}$.
\item If $2\le r \le s$, then $\LP(\ell_s(L_r))=\{2, r, s\}$.
\item If $s \le r \le 2$, then $\LP(\ell_s(L_r))=[r,2] \cup\{s\}$.
\item If $\max\{2,s\} \le r$, then $\LP(\ell_s(L_r))=\{2, r, s\}$.
\end{itemize}
\end{theorem}

Y.\@ Raynaud \cite{Raynaud1985} computed $\LP(L_s(L_r))$ in the case when $1\le r<s<\infty$. The following lemma, whose proof follows ideas from \cite{Raynaud1986}, aims to extend this result to the case when $0<r<s<\infty$.

\begin{lemma}\label{lem:RaynaudQB}
Let $0<s<\infty$ and $\XX$ and $\YY$ be quasi-Banach spaces. Suppose that $\YY$ has a complete minimal system $(\yy_n)_{n\in\Nt}$ and that $\YY\sqsubseteq L_s(\XX)$. Then, either $s\in\LP(\YY)$ or $\YY\sqsubseteq L_r(\XX)$ for all $r\in(0,s)$.
\end{lemma}

\begin{proof}
Assume without loss of generality that $\XX$ is a $q$-Banach space for some $0<q\le s$, so that $L_s(\XX)$ is a $q$-Banach space as well. Suppose there is $0<r<s$ such that $\YY\not\sqsubseteq L_r(\XX)$. Let $J$ be the inclusion of $L_s(\XX)$ into $L_r(\XX)$. Let $S\colon \YY \to L_s(\XX)$ be an isomorphic embedding. Given $F\subset \Nt$ finite, let $I_F$ the inclusion of
\[
\YY_F=[\yy_n \colon n\in\Nt \setminus F]
\]
into $\YY$. Assume by contradiction that $J\circ S \circ I_F$ is an isomorphic embedding. Since $\codim(\YY/\YY_F)=\abs{F}$,
\[
\YY\simeq \FF^{\abs{F}} \oplus \YY_F \sqsubseteq \FF^{\abs{F}}\oplus L_r(\XX) \simeq L_r(\XX).
\]
This absurdity shows that for any $F\subset\Nt$ finite and any $\varepsilon>0$ there is
\[
x\in S_\YY \cap \spn(\yy_n \colon n\in\Nt\setminus F)
\]
such that $\norm{J(S(x))}\le \varepsilon$. This fact allows us to recursively construct a pairwise disjoint sequence $(F_j)_{j=1}^\infty$ in $\Nt$ and a sequence $\XB=(x_j)_{j=1}^\infty$ in $S_\YY$ such that $x_j\in\spn(\yy_n \colon n\in F_n)$ for all $j\in\NN$, and
\[
\lim_j \norm{J(S(x_j))}=0.
\]
The sequence $\UB:=(S(x_j))_{j=1}^\infty$ converges to zero in measure. Since $\XB$ is a minimal system of $\YY$, $\UB$ is a minimal system of $\UU:=[S(x_j) \colon j\in\NN]$. Let $(\uu_j^*)_{j=1}^\infty$ be coordinate functionals for $\UB$ within $\UU$. Pick $(\varepsilon_j)_{j=1}^\infty$ in $(0,\infty)$ with
\[
\sum_{j=1}^\infty \norm{\uu_j^*}^q \varepsilon_j^q<\infty.
\]
By the gliding hump technique for the convergence in measure (see e.g.\@ \cite{AlbiacKalton2016}*{Lemma 5.2.1}), passing to a subsequence we can assume that there is a pairwise disjointly supported sequence $\VB:=(\vv_j)_{j=1}^\infty$ in $L_s(\XX)$ such that $\norm{S(x_j)-\vv_j}\le \varepsilon_j$ for all $j\in\NN$. By Lemma~\ref{lem:SPP}, $\XB$ and $\VB$ are equivalent. Hence, $\VB$ is semi-normalized. Therefore, $\VB$ is equivalent to the unit vector system of $\ell_s$.
\end{proof}

The next result complements that obtained by Y.\@ Raynaud \cite{Raynaud1985} and gives a positive answer to \cite{AnsorenaBello2025}*{Question 5.3}. Unlike the authors of these papers, we include locally non-convex spaces in our statement.

\begin{theorem}\label{thm:lpEmbedLsLr}
Let $r$, $s\in (0,\infty)$.
\begin{itemize}
\item If $r\le s \le 2$, then $\LP(L_s(L_r))=[r,2]$.
\item If $r\le 2 \le s$, then $\LP(L_s(L_r))=[r,2] \cup\{s\}$.
\item If $2\le r \le s$, then $\LP(L_s(L_r))=\{2, r, s\}$.
\item If $s \le \min\{2,r\}$, then $\LP(L_s(L_r))=[s,\max\{2,r\}]$.
\item If $2<s \le r$, then $\LP(L_s(L_r))=\{2\} \cup [s,r]$.
\end{itemize}
\end{theorem}

\begin{proof}
Suppose that $r\le s$. Since $L_r(L_r)\simeq L_r$, applying Lemma~\ref{lem:RaynaudQB} with $\XX=L_r$ gives $\LP(L_s(L_r))\subset \LP(L_s) \cup \LP(L_r)$. This inclusion yields the desired identity in all instances where $r\le s$. In turn, the instances where $r>s$ follow from combining \cite{AnsorenaBello2025}*{Proposition 5.2} with Theorem~\ref{lem:lpEmbeds}.
\end{proof}

To settle the embeddability of $\ell_p$-spaces into $L_s(\ell_r)$-spaces we need to develop new techniques.

\begin{lemma}\label{lem:A}
Let $\XX$ be a quasi-Banach space with a Schauder basis $\XB=(\xx_n)_{n=1}^\infty$, and $\LL$ be an absolutely continuous quasi-Banach function space over a (nonnull) $\sigma$-finite measure space $(\Omega,\Sigma,\mu)$. Let $(y_n)_{n=1}^\infty$ be a sequence in $\YY:=\LL(\XX)$. Suppose that
\[
\VV:=\enbrak{y_n \colon n\in\NN}\not\sqsubseteq \LL^m
\]
for all $m\in\NN $. Set for each $k$, $m\in\NN_0$ with $ k< m$
\[
\YY_{k,m}=\LL\enpar{\enbrak{\xx_n \colon k<n\le m}}, \quad \VV_{k,m}=\enbrak{y_n \colon k<n\le m}.
\]
Then, there are equivalent sequences $(u_j)_{j=1}^\infty$ and $(v_j)_{j=1}^\infty$ in $\YY\setminus\{0\}$, and increasing sequences $(k_j)_{k=0}^\infty$ and $(m_j)_{j=0}^\infty$ in $\ZZ$ such that $k_0=m_0=0$, and $u_j\in \YY_{m_{j-1},m_j}$ and $v_j\in \VV_{ k_{j-1}, k_j}$ for all $j\in \NN$.
\end{lemma}

\begin{proof}
Assume without loss of generality that $\XX$ is a $q$-Banach space, $0<q\le 1$, and $\XB$ is a bimonotone basis of $\XX$. Set for each $m\in\NN_0$
\[
\VV_{m}^{\,0}=\spn(y_n \colon n>m), \quad \VV_m=\enbrak{y_n \colon n>m}.
\]
Let $S_m\colon \XX \to \XX$ is the $m$th partial sum operator relative to $\XB$. Let $P_m \colon\YY\to\YY$ be the operator given by
\[
P_m(f)(\omega)=S_m(f(\omega)), \quad f\in\YY, \, \omega\in\Omega.
\]
Since $\LL$ is absolutely continuous,
\begin{enumerate}[label=(\Alph*),leftmargin=*, widest=a]
\item\label{it:SOT} $\lim_m P_m=\Id_{\YY}$ in the strong operator topology.
\end{enumerate}
Fix $k$, $m\in\NN_0$. Assume by contradiction that $P_m|_{\VV_k}$ is an isomorphic embedding. Since $d(k):=\codim(\VV/\VV_k)\in\NN_0$,
\[
\VV \sqsubseteq \FF^{d(k)} \oplus \LL^m \sqsubseteq \LL^{m+d(k)}.
\]
This absurdity shows that
\begin{enumerate}[label=(\Alph*),leftmargin=*, widest=a,resume]
\item\label{it:small} for all $\varepsilon>0$ and $k$, $m\in\NN_0$ there is $h\in S_\YY\cap \VV_{k}^{\,0}$ with $\norm{P_m(h)}\le \varepsilon$.
\end{enumerate}
Choose a nonincreasing sequence $(\varepsilon_j)_{j=1}^\infty$ in $(0,2^{-1/q})$ with
\[
\sum_{j=1}^\infty \frac{2\varepsilon_j^q}{1-2 \varepsilon_j^q}< 1.
\]
Use \ref{it:SOT} and \ref{it:small} to recursively construct increasing sequences $(k_j)_{k=0}^\infty$ and $(m_j)_{j=0}^\infty$ in $\ZZ$ and a sequence $\VB:=(v_j)_{j=1}^\infty$ in $S_\YY$ such that $k_0=m_0=0$, and
\[
\norm{P_{m_{j-1}}(v_j)} \le \varepsilon_j, \quad \norm{v_j -P_{m_j}(v_j)} \le \varepsilon_j,
\]
and $v_j\in \YY_{ k_{j-1}, k_j}$ for all $j\in\NN$. Define $\UB=(u_j)_{j=1}^\infty$ by
\[
u_j= P_{m_j}(v_j) - P_{m_{j-1}}(v_j), \quad j\in\NN.
\]
Since $\norm{u_j-v_j}^q \le 2 \varepsilon_j^q$, $\norm{u_j}^q\ge 1 - 2 \varepsilon_j^q>0$ for all $j\in\NN$. Since, $u_j\in \YY_{m_{j-1},m_j}\setminus\{0\}$ for all $j\in\NN$, $\UB$ is a bimonotone basis of $\UU:=[\UB]$. Hence, there is a sequence $(u_j^*)_{j=1}^\infty$ of coordinate functionals for $\UB$ within $\UU$ such that
\[
\norm{u_j^*} \norm{u_j}=1, \quad j\in\NN.
\]
Hence, $\UB$ and $\VB$ are equivalent by Lemma~\ref{lem:SPP}.
\end{proof}

For further reference, we record a classical result.

\begin{theorem}[see \cite{Pel1960}]\label{thm:Pel}
Let $1<p<\infty$. Then, $L_p(\ell_2)\simeq L_p$ and $B_{2,p} \simeq \ell_p$.
\end{theorem}

\begin{theorem}\label{thm:Lplq}
Let $0< r <s<\infty$.
\begin{itemize}
\item If $s\ge 2$, then $\LP( L_s(\ell_r))=A_{r,s}:=\{2,r,s\}$.
\item If $s\le 2$, then $\LP( L_s(\ell_r))=A_{r,s}:=\{r\} \cup [s,2]$.
\end{itemize}
\end{theorem}

\begin{proof}
By Corollary~\ref{cor:AMS+Paley} and Theorem~\ref{lem:lpEmbeds}, $A_{r,s}\subset \LP( L_s(\ell_r))$. Pick $p\in(0,\infty]\setminus\{2,s\}$ if $s\le 2$, and $p\in(0,\infty]\setminus[2,s]$ if $s\ge 2$. Assume that $\ell_p\sqsubseteq L_s(\ell_r)$, with the convention that $\ell_\infty$ means $c_0$. Since $\ell_p \not\sqsubseteq L_s$ by Corollary~\ref{cor:AMS+Paley}, an application of Lemma~\ref{lem:A} gives $\VB=(v_j)_{j=1}^\infty$ in $L_s(\ell_r)$ and an increasing sequence $(m_j)_{j=0}^\infty$ in $\ZZ$ with $m_0=0$ such that
\[
\supp(v_j)\subseteq [0,1]\times \enpar{(m_{j-1},m_j]\cap\NN}, \quad j\in \NN,
\]
and $\VB$ is equivalent to a semi-normalized block basic sequence of unit vector system of $\ell_p$. Since the unit vector system of $\ell_p$ is perfectly homogeneous, $\VB$ is equivalent to the unit vector system of $\ell_p$. Set $t=2/r$. Since $\VB$ is disjointly supported, applying Lemma~\ref{lem:DisjSup} gives
\[
\ell_{tp} \sqsubseteq \XX:=L_{ts} (\ell_{tr}).
\]
Since $ts > tr=2$, $\XX\simeq L_{st}$ by Theorem~\ref{thm:Pel}. Consequently, by Corollary~\ref{cor:AMS+Paley}, $tp\in\{ts,tr\}$. Hence, $p=r$.
\end{proof}
\section{On the isomorphic classification of mixed-norm Lebesgue spaces}\label{sect:iso}\noindent
In this section we will focus on locally convex spaces. Given $r$, $s\in[1,\infty]$ we consider the set
\[
\Rt_{r,s}=\enbrace{ L_s(L_r), \ell_s(L_r),L_s(\ell_r), Z_{r,s}, B_{r,s}}
\]
consisting of all mixed-norm Lebesgue spaces associated with the pair $(r,s)$. It is known \cite{AnsorenaBello2025} that if $\XX\in \Rt_{r,s}$, $\YY\in\Rt_{p,q}$ and $\XX\simeq\YY$, then either $(p,q)=(r,s)$ or $1<q=s<\infty$ and $\{p,r\}=\{2,s\}$. If we focus on separable spaces, we must take into account the well-known isomorphisms $\ell_2\simeq L_2$, $\ell_p\simeq B_{2,p}$ and $L_p\simeq L_p(L_2)$ for all $1<p<\infty$ (see Theorem~\ref{thm:Pel}). As non separable spaces are concerned, it is known that $\ell_\infty=L_\infty$ (see \cite{Pel1958}) and that $B_{s,\infty} \simeq Z_{s,\infty}$ for all $0<s<\infty$ (see \cite{AlbiacAnsorena2017}*{Proposition 4.2}). Techniques related to those leading to the latter isomorphism yield other isomorphisms between mixed-norm spaces, which we record below.

\begin{lemma}\label{lem:entriestosum}
Let $\Xt$ and $\Yt$ be countable infinite families of quasi-Banach spaces with modulus of concavity uniformly bounded. Let $\LL$ denote the lattice $c_0$ or $\ell_s$, $0<s\le\infty$. If $\Xt \trianglelefteq \Yt \trianglelefteq \Xt$, then
\[
\LL(\LL(\Xt))\simeq \LL(\Xt) \simeq \LL(\Yt).
\]
Besides, if a quasi-Banach space $\XX$ complementably embeds into $\LL(\Xt)$, then $\LL(\XX)\trianglelefteq\LL(\Xt)$.
\end{lemma}

\begin{proof}
Given a family $\Zt=(Z_n)_{n\in\Nt}$, we define $\Zt'=(Z_{n,k})_{(n,k)\in\Nt\times\NN}$ by $Z_{n,k}=Z_n$ for all $n\in\Nt$ and all $k\in\NN$. Since $\Xt\trianglelefteq\Xt$, $\Xt'\trianglelefteq\Xt\trianglelefteq\Xt'$. Hence, by the symmetry of $\LL$ we have
\[
\LL(\Xt')\trianglelefteq\LL(\Xt)\trianglelefteq\LL(\Yt)\trianglelefteq\LL(\Xt)\trianglelefteq\LL(\Xt').
\]
Symmetry also implies that $\LL(\Xt'')$ is naturally isomorphic to $\LL(\Xt')$. In turn, since $\LL(\LL)$ is lattice isomorphic to $\LL$, $\LL(\LL(\Xt))\simeq\LL(\Xt')$ and $\LL(\LL(\Xt'))\simeq\LL(\Xt'')$. Summing up, an application of Pe{\l}czy\'{n}ski decomposition method gives
\[
\LL(\Yt)\simeq\LL(\Xt)\simeq\LL(\Xt').
\]
Since it is clear that $\LL(\XX)\trianglelefteq\LL(\LL(\Xt))$, we are done.
\end{proof}

A quasi-Banach space $\XX$ is said to be a \emph{pseudo-dual} space if there is a Hausdorff vector topology on $\XX$, weaker than the norm-topology, relative to which the unit ball is relatively compact. Any dual space is a pseudo-dual by the Banach--Alaoglu Theorem.

\begin{lemma}\label{lem:elloosums}
Let $(Q_n)_{n=1}^\infty$ be a sequence of linear operators on a quasi-Banach space $\XX$. Assume that
\begin{itemize}
\item $Q_n\circ Q_k=Q_{\min\{n,k\}}$ for all $n$, $k\in\NN$,
\item $\lim Q_n=\Id_\XX$ in the strong topology of operators, and
\item $\XX \trianglelefteq \YY$ for some pseudo-dual space $\YY$.
\end{itemize}
Then, $\ell_\infty(\XX)\simeq \enpar{\oplus_{n=1}^\infty Q_n(\XX)}_{\ell_\infty}$.
\end{lemma}

\begin{proof}
Let $L\colon \XX \to \UU:=\enpar{\oplus_{n=1}^\infty Q_n(\XX)}_{\ell_\infty}$ be the linear bounded map given by
\[
L(f)=\enpar{Q_n(f)}_{n=1}^\infty, \quad f\in\XX.
\]
Let $J\colon \XX\to \YY$ and $P\colon \YY\to\XX$ be linear bounded maps such that $P\circ J=\Id_{\XX}$. Let $\tau$ be a topology on $\YY$ that witnesses $\YY$ is a pseudo-dual. Pick a nonprincipal ultrafilter $\Ut$ over $\NN$. Define
\[
T\colon \UU\to \YY, \quad (f_n)_{n=1}^\infty\mapsto \tau\mbox{--}\lim_{n,\Ut} J(f_n).
\]
It is routine to check that $P\circ T \circ L=\Id_{\XX}$. So, $\XX \trianglelefteq \UU$. Hence, by Lemma~\ref{lem:entriestosum}, $\VV:=\ell_\infty(\XX)\trianglelefteq \UU$ and $\ell_\infty(\VV)\simeq\VV$. Since, by construction, $\UU \trianglelefteq \VV$, applying the Pe{\l}czy\'{n}ski decomposition technique gives $\UU\simeq\VV$.
\end{proof}

\begin{theorem}
$L_\infty(L_s)\simeq L_\infty(\ell_s)$ and $\ell_\infty(L_s)\simeq Z_{s,\infty}\simeq B_{s,\infty}$ for all $s\in[1,\infty)$.
\end{theorem}

\begin{proof}
With the only exception of $L_1$, $\ell_s$ and $L_s$ are dual spaces. In turn, by the Radon--Nikodym theorem, $L_1$ is a complemented subspace to the space of signed measures on $[0,1]$, which is a dual space. For each $n\in\NN$, let $Q_n\colon L_s\to L_s$ be the conditional expectation operator associated with the $n$th step of the dyadic filtration. By Lemma~\ref{lem:elloosums},
\[
\ell_\infty(L_s) \simeq \enpar{\oplus_{n=1}^\infty Q_n(L_s)}_{\ell_\infty}\simeq \enpar{\oplus_{n=1}^\infty \ell_s^{2^n}}_{\ell_\infty}.
\]
Similarly, applying Lemma~\ref{lem:elloosums} to the canonical partial sum projections on $\ell_s$ yields
\[
\ell_\infty(\ell_s)\simeq \enpar{\oplus_{n=1}^\infty \ell_s^{n}}_{\ell_\infty}.
\]
By Lemma~\ref{lem:entriestosum}, $\ell_\infty(L_s)\simeq Z_{s,\infty} \simeq B_{s,\infty}$. Consequently,
$L_\infty(\ell_\infty(L_s))\simeq L_\infty(Z_{s,\infty})$. Since $L_\infty(\ell_\infty)$ is lattice isomorphic to $L_\infty$, we are done.
\end{proof}

Considering the knowledge we have gathered, in order to classify by isomorphism the Banach spaces in
\begin{equation}\label{eq:cup.Rrs}
\bigcup_{1\le r,s\le\infty} \Rt(r,s),
\end{equation}
we have to classify the spaces in each of the following families:
\begin{itemize}
\item $\Et_1=\{\ell_1,L_1\}$,
\item $\Et_{2,s}=\{L_s(L_2),Z_{2,s}, B_{2,s}\}$ for $s\in[1,\infty)\setminus\{2\}$,
\item $\Et_{r,s}=\Rt_{r,s}$ for $r\in[1,\infty)\setminus\{2\}$ and $s\in[1,\infty)$ with $r\not=s$,
\item $\Et_{\infty,s}=\{L_s(L_\infty), Z_{\infty,s}, B_{\infty,s}\}$ for $s\in[1,\infty)$,
\item $\Et_{r,\infty}=\enbrace{ L_\infty(L_r), B_{r,\infty}}$ for $r\in[1,\infty)$.
\end{itemize}

It is well-known that the two spaces in $\Et_1$ are not isomorphic. Indeed, we could use Theorem~\ref{thm:AMS} and Theorem~\ref{thm:PelSti} to tell apart $\ell_1$ from $L_1$. As $\Et_{\infty,s}$, $1\le s<\infty$, is concerned, we note that $B_{\infty,s}$ is separable, and the non-separable spaces $L_s(L_\infty)$ and $\ell_s(\ell_\infty)$ are not isomorphic by Example~\ref{ex:Kania}. So, the three spaces in $\Et_{\infty,s}$, $1\le s<\infty$, are mutually non-isomorphic. We will use the results achieved in Section~\ref{sect:lpembeds} to address the classification by isomorphism of the spaces in each family $\Et_{r,s}$ for $r$, $s\in[1,\infty)$, $r\not=s$.

\begin{theorem}\label{thm:IsoClas}
For $r$, $s\in(1,\infty)$, $r\not=s$, the Banach spaces in $\Et_{r,s}$ are mutually non-isomorphic. Concerning the families $\Et_{r,s}$ with $1=\min\{r,s\}<\max\{r,s\}<\infty$, we have the following.
\begin{itemize}
\item Let $1<s<\infty$. If two different spaces in $\Et_{1,s}$ are isomorphic, these spaces are $L_s(L_1)$ and $\ell_s(L_1)$.
\item Let $1<r<\infty$. If two different spaces in $\Et_{r,1}$ are isomorphic, these spaces are $L_1(L_r)$ and $L_1(\ell_r)$.
\end{itemize}
\end{theorem}

\begin{proof}
Combining Theorem~\ref{thm:LPBesov}, Theorem~\ref{thm:lplq}, Theorem~\ref{thm:lpLq}, Theorem~\ref{thm:lpEmbedLsLr} and Theorem~\ref{thm:Lplq}, and using that if $\XX\simeq\YY$ then $\LP(\XX)\simeq\LP(\YY)$ and $\LP(\XX^*)\simeq\LP(\YY^*)$, tells apart all the required spaces but $L_2(\ell_1)$ from $\ell_2(\ell_1)$ (see Tables~\ref{table:1}, \ref{table:2}, \ref{table:3}, \ref{table:4} and \ref{table:5}). To complete the proof, use Example~\ref{ex:Diestel} or Example~\ref{ex:Dilworth}.
\end{proof}

\begin{table}
\begin{tabular}{|c||c | c | c|}
\hline
\multicolumn{4}{|c|}{$\Et_{2,s}$ for $1\le s<2$} \\
\hline
$ \XX $&$ L_s(L_2) $&$ Z_{2,s} $&$ B_{2,s} $\\
\hline
$ \LP(\XX) $&$ [s,2] $&$ \{2,s\} $&$ \{s\} $\\
\hline
\hline
\multicolumn{4}{|c|}{$\Et_{2,s}$ for $2<s<\infty$} \\
\hline
$ \XX $&$ L_s(L_2) $&$ Z_{2,s} $&$ B_{2,s} $\\
\hline
$ \LP(\XX) $&$ \{2,s\} $&$ \{2,s\} $&$\{s\} $\\
\hline
\end{tabular}
\medskip
\caption{}
\label{table:1}
\end{table}

\begin{table}
\begin{tabular}{|c||c|c|c|c|c|}
\hline
\multicolumn{6}{|c|}{$\Et_{r,s}$ for $1\le r < s\le 2$} \\
\hline
$ \XX $&$ L_s(L_r) $&$ \ell_s(L_r) $&$ L_s(\ell_r) $&$ Z_{r,s} $&$ B_{r,s} $\\
\hline
$ \LP(\XX) $&$ [r,2] $&$ [r,2] $&$ \{r\} \cup[s,2] $&$ \{r,s\} $&$ \{s\} $\\
\hline
\hline
\multicolumn{6}{|c|}{$\Et_{r,s}$ for $2\le s<r<\infty$} \\
\hline
$ \XX $&$ L_s(L_r) $&$ \ell_s(L_r) $&$ L_s(\ell_r) $&$ Z_{r,s} $&$ B_{r,s} $\\
\hline
$ \LP(\XX) $&$ \{2\} \cup [s,r] $&$ \{2,r,s\} $&$ \{2,r,s\} $&$ \{r,s\} $&$ \{s\} $\\
\hline
\end{tabular}
\medskip
\caption{}
\label{table:2}
\end{table}

\begin{table}
\begin{tabular}{|c||c|c|c|c|c|}
\hline
\multicolumn{6}{|c|}{$\Et_{r,s}$ for $1\le r < 2\le s<\infty$} \\
\hline
$ \XX $&$ L_s(L_r) $&$ \ell_s(L_r) $&$ L_s(\ell_r) $&$ Z_{r,s} $&$ B_{r,s} $\\
\hline
$ \LP(\XX) $&$ [r,2] \cup\{s\} $&$ [r,2] \cup\{s\}$&$ \{r,2,s\} $&$ \{r,s\} $&$ \{s\} $\\
\hline
\hline
\multicolumn{6}{|c|}{$\Et_{r,s}$ for $1\le s\le 2<r<\infty$} \\
\hline
$ \XX $&$ L_s(L_r) $&$ \ell_s(L_r) $&$ L_s(\ell_r) $&$ Z_{r,s} $&$ B_{r,s} $\\
\hline
$ \LP(\XX) $&$ [s,r] $&$ \{2,r,s\} $&$ \{2,r,s\} $&$ \{r,s\} $&$ \{s\} $\\
\hline
\end{tabular}
\medskip
\caption{}
\label{table:3}
\end{table}

\begin{table}
\begin{tabular}{|c||c|c|c|c|c|}
\hline
\multicolumn{6}{|c|}{$\Et_{r,s}$ for $1\le s < r < 2$} \\
\hline
$ \XX $&$ L_s(L_r) $&$ \ell_s(L_r) $&$ L_s(\ell_r) $&$ Z_{r,s} $&$ B_{r,s} $\\
\hline
$ \LP(\XX) $&$ [s,2] $&$ \{s\} \cup [r,2] $&$ [s,2] $&$ \{r,s\} $&$ \{s\} $\\
\hline
\hline
\multicolumn{6}{|c|}{$\Et_{r,s}$ for $2 < r<s<\infty$} \\
\hline
$ \XX $&$ L_s(L_r) $&$ \ell_s(L_r) $&$ L_s(\ell_r) $&$ Z_{r,s} $&$ B_{r,s} $\\
\hline
$ \LP(\XX) $&$ \{2\} \cup [r,s] $&$ \{2,r,s\} $&$ \{2,r,s\} $&$ \{r,s\} $&$ \{s\} $\\
\hline
\end{tabular}
\medskip
\caption{}
\label{table:4}
\end{table}

\begin{table}
\begin{tabular}{|c||c|c|c|c|c|}
\hline
\multicolumn{6}{|c|}{$\Et_{r,s}$ for $1\le s \le 2<r$} \\
\hline
$ \XX $&$ L_s(L_r) $&$ \ell_s(L_r) $&$ L_s(\ell_r) $&$ Z_{r,s} $&$ B_{r,s} $\\
\hline
$ \LP(\XX) $&$ [s,r] $&$ \{2,r,s\} $&$ [s,r] $&$ \{r,s\} $&$ \{s\} $\\
\hline
\hline
\multicolumn{6}{|c|}{$\Et_{r,s}$ for $r < 2 \le s<\infty$} \\
\hline
$ \XX $&$ L_s(L_r) $&$ \ell_s(L_r) $&$ L_s(\ell_r) $&$ Z_{r,s} $&$ B_{r,s} $\\
\hline
$ \LP(\XX) $&$ \{s\} \cup [r,2] $&$ \{s\} \cup [r,2] $&$ \{2,r,s\} $&$ \{r,s\} $&$ \{s\} $\\
\hline
\end{tabular}
\medskip
\caption{}
\label{table:5}
\end{table}

We close the paper writting down the problems that Theorem~\ref{thm:IsoClas} leaves open in order to totally classify, up to isomorphism, the spaces listed in \eqref{eq:cup.Rrs}.

\begin{question}
Let $1<s<\infty$. Are $L_s(L_1)$ and $\ell_s(L_1)$ isomorphic?
\end{question}

\begin{question}
Let $1<r<\infty$, $r\not=2$. Are $L_1(L_r)$ and $L_1(\ell_r)$ isomorphic?
\end{question}

\begin{question}\label{qt:CM}
Let $r\in[1,\infty)$. Are the spaces $L_\infty(L_r)$ and $B_{r,\infty}$ isomorphic?
\end{question}

Question~\ref{qt:CM} is framed within the general problem of studying the class of Banach spaces $\XX$ for which $L_\infty(\XX)$ and $\ell_\infty(\XX)$ are isomorphic (see \cite{CembranosMendoza1997}*{Problem 7.6} or \cite{Rodriguez2017}*{Problem 5.7}).
\section*{Acknowledgements}
The authors are indebted to S.\@ Dilworth and J.\@ Rodr\'{\i}guez for drawing the papers \cites{Dilworth1990,Diestel1977} to the their attention.
\bibliography{Biblio}
\bibliographystyle{plain}
\end{document}